\documentclass[
final
]{dmtcs-episciences}

\usepackage[utf8]{inputenc}
\usepackage{latexsym,algorithm2e,enumerate}
\usepackage{amsmath,amssymb,amsthm,mathrsfs,amsfonts,amscd}
\usepackage{tabularx}
\usepackage{multirow}
\usepackage{booktabs}
\usepackage{tikz}
\usetikzlibrary{shapes}
\usetikzlibrary{shapes.geometric}
\usetikzlibrary{positioning}
\usetikzlibrary{graphs}
\usetikzlibrary{arrows}
\usetikzlibrary{calc}
\usetikzlibrary{decorations.markings}
\usetikzlibrary{fit}
\usetikzlibrary{backgrounds}
\usetikzlibrary{patterns}
\tikzset{>=latex}
\usepackage{amsmath,amsthm,amsfonts,amssymb,amscd}
\usepackage{graphicx}

\newtheorem{theorem}{Theorem}[section]

\newtheorem{lemma}{Lemma}[section]
\newtheorem{proposition}{Proposition}[section]
\newtheorem{corollary}{Corollary}[section]

\usepackage[round]{natbib}

\author[M.C. Dourado \and M. Gutierrez \and F. Protti \and S. Tondato]{Mitre C. Dourado\affiliationmark{1}
\and Marisa Gutierrez\affiliationmark{2,4}
\and F\'abio Protti\affiliationmark{3}
\and Silvia Tondato\affiliationmark{2}}
\title[Weakly toll convexity and proper interval graphs]{Weakly toll convexity and  proper interval graphs}
\affiliation{
Instituto de Computa\c c\~ao, Universidade Federal do Rio de Janeiro, Brazil\\
CMaLP Facultad de Ciencias Exactas, Universidad Nacional de La Plata, Argentina\\
Instituto de Computa\c c\~ao, Universidade Federal Fluminense, Niter\'oi, RJ, Brazil\\
Conicet, Argentina }

\keywords{Convex Geometry, Convexity, Proper Interval Graph,
Weakly Toll Walk}

\publicationdata{vol. 26:2}{2024}{4}{10.46298/dmtcs.9837}{2022-07-26; 2022-07-26; 2023-03-13; 2023-07-22; 2023-12-21; 2024-03-19}{2024-03-27}

\begin{document}
\maketitle
\begin{abstract}
{\em Abstract.} A walk $u_0,u_1,\ldots,u_{k-1},u_k$ in a graph $G$
is a \textit{weakly toll walk} if $u_0u_i \in E(G)$ implies $u_i =
u_1$ and $u_ju_k\in E(G)$ implies $u_j=u_{k-1}$. A set $S$ of
vertices of $G$ is {\it weakly toll convex} if for any two
nonadjacent vertices $x,y \in S$, any vertex in a weakly toll walk
between $x$ and $y$ is also in $S$. The {\em weakly toll
convexity} is the graph convexity defined over weakly toll convex
sets. If $S$ and $S\setminus\{x\}$ are convex sets, then $x$ is an
{\em extreme vertex} of $S$. A graph convexity is said to be a
{\em convex geometry} if it satisfies the Minkowski-Krein-Milman
property, which states that every convex set is the convex hull of
its extreme vertices. It is known that chordal, Ptolemaic, weakly
polarizable, and interval graphs can be characterized as convex
geometries with respect to the monophonic, geodesic, $m^3$, and
toll convexities, respectively. Inspired by previous results
in~\citep*{lm}, in this paper we prove that a graph is a convex
geometry with respect to the weakly toll convexity if and only if
it is a proper interval graph. Furthermore, some well-known graph
invariants are studied with respect to the weakly toll convexity,
namely the weakly toll number and the weakly toll hull number. In
particular, we determine these invariants for trees and we find
bounds for interval graphs.
\end{abstract}

\section{Introduction}

This paper is motivated by the results and ideas contained
in~\citep*{lm}. We introduce a new graph convexity and show how
this gives rise to a new structural characterization of proper
interval graphs. We begin with an overview of convexity notions in
graphs. For an extensive overview of other convex structures,
see~\citep*{v}.

Let $\mathcal{C}$ be a collection of subsets (called convex sets)
of a finite set $X$. In abstract convexity theory, the following
axioms determine the pair $(X, \mathcal{C})$ as a convexity space:
(i) $\emptyset$ and $X$ are convex; (ii) the intersection of any
two convex sets is convex. Suppose that $X=V(G)$ for some graph
$G$. For a set $S \subseteq V(G)$, the smallest convex set
containing $S$ is called the \textit{convex hull} of $S$. A set $S
\subseteq V(G)$ is a \textit{hull set} if the convex hull of $S$
is $V(G)$. An element $x \in S$, where $S \subseteq V(G)$ is a
convex set, is called an \textit{extreme vertex} of $S$ if
$S\setminus\{x\}$ is also convex. A \textit{convex geometry} is a
pair formed by a graph $G$ and a convexity on $V(G)$ satisfying
the \textit{Minkowski-Krein-Milman property}~\citep*{km}:
\textit{Every convex set is the convex hull of its extreme
vertices.}

In the last few decades, graph convexity has been investigated in
many contexts~\citep*{fj,p,v}. In particular, some studies are
devoted to determine if a graph equipped with a convexity is a
convex geometry. Chordal, Ptolemaic, strongly chordal, interval,
and weakly polarizable graphs have been characterized as convex
geometries with respect to the monophonic~\citep*{du,fj},
geodesic~\citep*{fj}, strong~\citep*{fj}, toll~\citep*{lm}, and
$m^3$~\citep*{dnb} convexities, respectively. Other classes of
graphs that have been characterized as convex geometries are
forests, cographs, bipartite graphs, and planar graphs (see the
survey paper~\citep*{dpsa}). In addition, convex geometries
associated with the {\em Steiner convexity} and the {\em
$\ell^k$-convexity} (defined over induced paths of length at most
$k$) are studied in~\citep*{caceres} and~\citep*{gpt},
respectively.

The main result of this paper states that a graph is a convex
geometry with respect to the {\em weakly toll convexity} if and
only if it is a {\em proper interval graph}. In order to prove
this result we introduce the concept of {\em weakly toll walk}, a
walk with a special restriction on their end vertices. A walk
$u_0, u_1, \ldots, u_{k-1}, u_k$ is a \textit{weakly toll walk} if
$u_0u_i \in E(G)$ implies $u_i = u_1$ and $u_ju_k\in E(G)$ implies
$u_j=u_{k-1}$. Note that $u_1$ (or $u_{k-1}$) may appear more than
once in the walk. A set $S$ of vertices of $G$ is {\it weakly toll
convex} if for any two nonadjacent vertices $x,y \in S$, any
vertex in a weakly toll walk between $x$ and $y$ is also in $S$.
The weakly toll convexity is the graph convexity defined over
weakly toll convex sets.

The concept of weakly toll walk is a relaxation of the concept of
{\em tolled walk}, which was conceived to capture the structure of
the convex geometry associated with an interval graph. Likewise,
weakly toll walks are used as a tool to characterize proper
interval graphs as convex geometries.

The paper is organized as follows: in Section 2, we give some
definitions and necessary background. In Section 3, we prove that
proper interval graphs are precisely the convex geometries with
respect to the weakly toll convexity. In Section 4, we study some
invariants associated with the weakly toll convexity, namely the
weakly toll number and the weakly toll hull number. In particular,
we determine these invariants in trees and we find bounds in
arbitrary interval graphs. Section 5 contains a short conclusion.

\section{Preliminaries}

Let $G$ be an undirected graph without loops or multiple edges. If
$C$ is a subset of vertices of $G$, $G[C]$ denotes the subgraph of
$G$ induced by $C$. Let $xy\in E(G)$ and $z,w$ be two nonadjacent
vertices of $G$; the graph $G-xy+zw$ is obtained from $G$ by
deleting the edge $xy$ and adding the edge $zw$. For $S\subseteq
V(G)$, the graph $G'=G-S$ is defined as follows:
$V(G')=V(G)\setminus S$ and $E(G')=\{xy\in E(G) \mid \{x,y\}\cap
S=\emptyset\}$.

The {\em distance} between a pair of vertices $u$ and $v$ of $G$
is the length of a shortest path (or geodesic) between $u$ and $v$
in $G$, and is denoted by $d_G(u, v)$. The \textit{geodesic
interval} $I_G(u, v)$ between vertices $u$ and $v$ is the set of
all vertices that lie in some shortest path between $u$ and $v$ in
$G$; in other words, $I_G(u, v)=\{x \in V(G):d_G(u, x)+d_G(x,
v)=d_G(u,v)\}$. A subset $S$ of $V(G)$ is \textit{geodesically
convex} (or \textit{g-convex}) if $I_G(u, v)\subseteq S$ for all
$u, v \in S$. Similarly, define $J_G(u, v) = \{x \in V(G) : x \
\text{lies in an induced path between} \ u \ \text{and} \ v\}$ to
be the \textit{monophonic interval} between $u$ and $v$ in $G$. In
the associated monophonic convexity, a subset $S$ of $V(G)$ is
\textit{monophonically convex} (or \textit{m-convex}) if $J_G(u,
v) \subseteq S$ for all $u,v \in S$.

A \textit{tolled walk} is any walk $T:u_0,u_1,\ldots,u_{k-1},u_k$
such that $u_0$ is adjacent only to the second vertex of the walk,
and $u_k$ is adjacent only to the second-to-last vertex of the
walk. This implies that each of $u_1$ and $u_{k-1}$ occurs exactly
once in the walk.

Let $T_G(u, v) = \{x \in V(G) : x \ \text{lies in a tolled walk
between} \ u \ \text{and} \ v\}$ be the \textit{toll interval}
between $u$ and $v$ in $G$. In the associated toll convexity, a
subset $S$ of $V(G)$ is \textit{toll convex} (or
\textit{t-convex}) if $T_G(u, v) \subseteq S$ for all $u,v\in S$.

Next, we introduce the concept of weakly toll convexity. Let
$u,v\in V(G)$. A {\em weakly toll walk} between $u$ and $v$ in $G$
is a sequence of vertices of the form

$$\mathit{W} : u=w_0,w_1,\ldots,w_{k-1},v=w_k, \ k\geq 0,$$

\noindent such that if $k>0$, then
\begin{itemize}
\item $w_iw_{i+1} \in E(G)$ for all $i\in\{0,\ldots,k-1\}$, \item
$uw_i \in E(G)$ implies $w_i=w_1$, $i\in\{1,\ldots,k\}$, and \item
$w_iv \in E(G)$ implies $w_i=w_{k-1}$, $i\in\{0,\ldots,k-1\}$.
\end{itemize}

In other words, a weakly toll walk is any walk
$\mathit{W}:u,w_1,\ldots,w_{k-1},v$ between $u$ and $v$ such that
$u$ is adjacent only to the vertex $w_1$, which can appear more
than once in the walk, and $v$ is adjacent only to the vertex
$w_{k-1}$, which can appear more than once in the walk. Note that
if $uv \in E(G)$, then $\mathit{W}:u,v$ is a weakly toll walk, and
if $k=0$ then $\mathit{W}:u$ is a wealky toll walk consisting of a
single vertex. We define $\mathit{WT}_G(u, v) = \{x \in V(G) : x \
\text{lies in a weakly toll walk between} \ u \ \text{and} \ v\}$
to be the \textit{weakly toll interval} between $u$ and $v$ in
$G$. Finally, a subset $S$ of $V(G)$ is \textit{weakly toll
convex} if $\mathit{WT}_G(u, v) \subseteq S$ for all $u, v \in S$.

Note that any weakly toll convex set is also a toll convex set.
Also, any toll convex set is a monophonically convex set, and a
monophonically convex set is a geodesically convex set.

On the other hand, consider the graph $K_{1,3}$ with vertices
$a,b,c,d$, where $b$ is the vertex with degree three, and let
$S=\{a,b,c\}$. It is clear that $S$ is toll convex but not weakly
toll convex, since $a,b,d,b,c$ is a weakly toll walk between $a$
and $c$ that contains $d \notin S$.

The {\em weakly toll convex hull} of a set $S \subseteq V(G)$ is
the smallest set of vertices in $G$ that contains $S$ and is
weakly toll convex (alternatively, it is the intersection of all
weakly toll convex sets that contain $S$). In this particular
convexity, the concept of extreme vertex can also be defined. A
vertex $x$ of a weakly toll convex set $S$ of a graph $G$ is an
\textit{extreme vertex} of $S$ if $S\setminus \{x\}$ is also a
weakly toll convex set in $G$.

A graph is an \textit{interval graph} if it has an intersection
model (or \textit{interval model}) consisting of closed intervals
on a straight line. Given an interval ${\mathcal I}$, let
$R({\mathcal I})$ and $L({\mathcal I})$ be, respectively, the
right and left endpoints of ${\mathcal I}$. Given a family of
intervals $\{{\mathcal I}_v\}_{v \in V(G)}$, we say that
${\mathcal I}_a$ is an \textit{end interval} if $L({\mathcal
I}_a)=\mathit{Min}\bigcup {\mathcal I}_v$ or $R({\mathcal
I}_a)=\mathit{Max}\bigcup {\mathcal I}_v$. A given vertex $a$ in
an interval graph $G$ is an \textit{end vertex} if there exists
some interval model where $a$ is represented by an end interval.

A graph is \textit{chordal} if every cycle of length at least four
has a chord. A vertex $x$ of a graph $G$ is called
\textit{simplicial} if $N[x]$ is a clique in $G$, where $N[x]=\{u
\in V(G): ux \in E(G)\}\cup \{x\}$. The ordering $x_1,\ldots, x_n$
of the vertices of $G$ is a \textit{perfect elimination order} of
$G$ if for all $i$, $x_i$ is simplicial in
$G[\{x_i,\ldots,x_n\}]$.

\begin{theorem}{\em \citep*{d}}
A graph is chordal if and only if it has a perfect elimination
order.
\end{theorem}

\citet*{lb} proved that a chordal graph is an interval graph if
and only if it contains no asteroidal triple. Three vertices of a
graph form an \textit{asteroidal triple} if for any two of them,
there exists a path between them that does not intersect the
closed neighborhood of the third.

\citet*{g} studied the end vertices of an interval graph:

\begin{theorem} {\em \citep*{g}}\label{tg}
A vertex $v$ of an interval graph $G$ is an end vertex if and only
if $G$ does not contain any of the graphs in Figure~\ref{grafica2}
as an induced subgraph with $v$ as the designated vertex.
\end{theorem}

A simplicial vertex $v$ of an interval graph $G$ is called
\textit{end simplicial vertex} if it is an end vertex of $G$.

In Figure~\ref{grafica2}, consider the graphs ${\mathit
star}_{1,2,2}$, $B_n \; (n> 5)$, and the bull graph. Note that $v$
is a simplicial vertex in any of such graphs. In addition,
consider the walks $v_1,v_2,v_3,v,v_3,v_4,v_5$ (in the graph
${\mathit star}_{1,2,2}$), $v_1,v_2,v,v_2,v_{n-1},v_n$ (in the
graph $B_n$), and $v_1,v_2,v,v_3,v_4$ (in the bull graph). Note
that these walks are weakly toll walks. Hence $v$ is not an
extreme vertex of any graph in Figure ~\ref{grafica2}.

\begin{figure}[ht]
\tikzstyle{miEstilo}= [thin, dotted]
\begin{center}
\begin{tikzpicture}[scale=0.7]
\node[draw,circle] (6) at (5,0) {$v_1$}; \node[draw,circle] (7) at
(7,0) {$v_2$}; \node[draw,circle] (8) at (9,0) {$v_3$};
\node[draw,circle] (9) at (11,0) {$v_4$}; \node[draw,circle] (10)
at (13,0) {$v_5$}; \node[draw,circle] (11) at (9,2) {$v$}; \draw
(6)--(7)--(8)--(9)--(10); \draw (8)--(11); \node (40) at (9,-1)
{${\mathit star}_{1,2,2}$}; \node[draw,circle] (12) at (5,-4)
{$v$}; \node[draw,circle] (13) at (7,-4) {$v_3$};
\node[draw,circle] (14) at (9,-4) {$v_4$}; \node[draw,circle] (15)
at (11,-4) {$v_5$}; \node[draw,circle] (16) at (13,-4)
{$v_{n-1}$}; \node[draw,circle] (17) at (15,-4) {$v_n$};
\node[draw,circle] (18) at (11,-2) {$v_2$}; \node[draw,circle]
(19) at (13,-2) {$v_1$}; \draw
(18)--(12)--(13)--(14)--(15)--(18)--(13); \draw
(19)--(18)--(16)--(17); \draw[miEstilo] (15)--(16); \draw
(18)--(14); \node (19) at (10,-5) {$B_n \; (n >5)$};
\node[draw,circle] (20) at (18,0) {$v_2$}; \node[draw,circle] (22)
at (21,0) {$v_3$}; \node[draw,circle] (21) at (18,-2) {$v_1$};
\node[draw,circle] (23) at (21,-2) {$v_4$}; \node[draw,circle]
(24) at (19.5,1) {$v$}; \node (25) at  (19.5, -2.5)
{$\mathit{bull}$}; \draw (21)--(20)--(22)--(23); \draw
(20)--(24)--(22); \node[draw,circle] (31) at (18,-4) {};
\node[draw,circle] (32) at (19,-4) {}; \node[draw,circle] (33) at
(20,-4) {$v$}; \node[draw,circle] (34) at (21,-4) {};
\node[draw,circle] (35) at (22,-4) {}; \draw
(31)--(32)--(33)--(34)--(35);
\end{tikzpicture}
\end{center}
\caption{Gimbel's graphs. \label{grafica2}}
\end{figure}
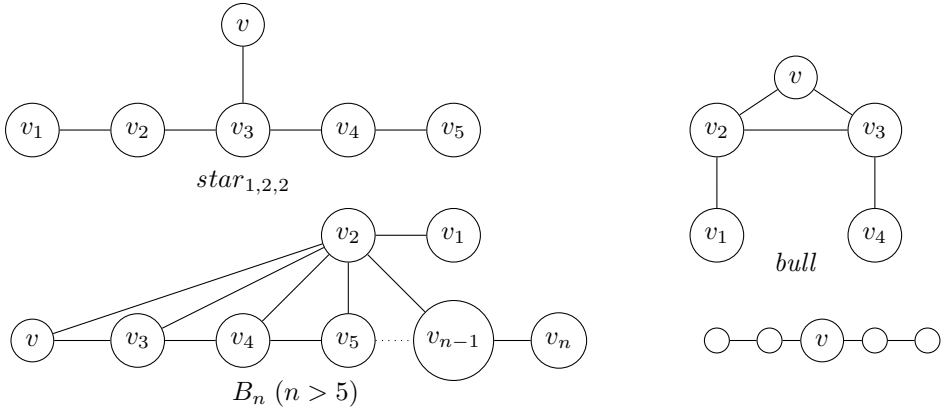

A \textit{proper interval graph} is an interval graph that has a
\textit{proper interval model}, i.e., an interval model in which
no interval strictly contains another. More precisely, there
exists a collection of intervals ${\mathscr I}=\{{\mathcal
I}_v\}_{v \in V(G)}$ such that:

\begin{enumerate}
\item $uv\in E(G)$ if and only if ${\mathcal I}_u \cap {\mathcal
I}_v\neq \emptyset$; \item ${\mathcal I}_v \nsubseteq {\mathcal
I}_u$ and ${\mathcal I}_u \nsubseteq {\mathcal I}_v$ if $u \neq
v$.
\end{enumerate}

\citet*{r} proved that proper interval graphs are exactly the
$K_{1,3}$-free interval graphs. Observe that $star_{1,2,2}$ and
$B_n$ (see Figure~\ref{grafica2}) are not proper interval graphs.
In addition, in a proper interval model of the graph $P_5$, the
end intervals are necessarily associated with the vertices of
degree one. Therefore, we have the following corollary:

\begin{corollary}\label{t1}
A vertex $v$ of a proper interval graph $G$ is an end vertex if
and only if $G$ does not contain any of the graphs in
Figure~\ref{grafica11} as an induced subgraph with $v$ as the
designated vertex.
\end{corollary}

\begin{figure}[ht]
\begin{center}
\begin{tikzpicture}[scale=0.5]
 \node[draw,circle] (1) at (-1,0) {$v_1$};
 \node[draw,circle] (2) at (1,0) {$v$};
 \node[draw,circle] (5) at (3,0) {$v_2$};
 \draw (1)--(2)--(5);
 \node[draw,circle] (20) at (13,0) {$v_2$};
 \node[draw,circle] (22) at (18,0) {$v_3$};
 \node[draw,circle] (21) at (13,-2) {$v_1$};
 \node[draw,circle] (23) at (18,-2) {$v_4$};
 \node[draw,circle] (24) at (15.5,1) {$v$};
 \node (25) at  (15.5, -3) {$\mathit{bull}$};
 \draw (21)--(20)--(22)--(23);
 \draw (20)--(24)--(22);
\end{tikzpicture}\end{center}
\caption{Forbidden configurations for an end vertex in proper
interval graphs.} \label{grafica11}
\end{figure}
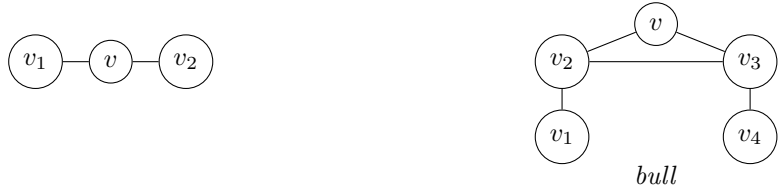

\if 10

\begin{figure}[ht]
\begin{center}
\begin{tikzpicture}[scale=0.5]
\node[draw,circle] (20) at (9,0) {$v_2$}; \node[draw,circle] (22)
at (13,0) {$v_3$}; \node[draw,circle] (21) at (9,-2) {$v_1$};
\node[draw,circle] (23) at (13,-2) {$v_4$}; \node[draw,circle]
(24) at (11.5,1) {$v$}; \node (25) at  (11.5, -3) {$bull$}; \draw
(21)--(20)--(22)--(23); \draw (20)--(24)--(22);
\end{tikzpicture}\end{center}
\caption{Forbidden configuration for an end simplicial vertex in
proper interval graphs.}\label{grafica12}
\end{figure}

\fi

In order to determine whether a graph is a convex geometry with
respect to some convexity, it is important to know which vertices
are the extreme vertices of convex sets in that convexity. As said
above, chordal graphs are associated with the monophonic
convexity, and in this case the extreme vertices are exactly the
simplicial vertices~\citep*{fj}. In the case of interval graphs,
associated with the toll convexity, the extreme vertices are the
end simplicial vertices~\citep*{lm}.

\begin{theorem}{\em \citep*{lm}} \label{ti}
A graph is a convex geometry with respect to the toll convexity if
and only if it is an interval graph.
\end{theorem}

\begin{lemma} \label{igw} Let $G$ be an interval graph. Every vertex of $G$ which is not an end simplicial vertex lies in a weakly toll walk between two end simplicial vertices.
\end{lemma}

\begin{proof} Since $G$ is an interval graph, by Theorem \ref{ti}, $G$ is a convex geometry with respect to the toll convexity. This means that every vertex of $G$ that is not an end simplicial vertex lies in a tolled walk between two end simplicial vertices. As every tolled walk is a weakly toll walk, the lemma follows.
\end{proof}

\begin{lemma}\label{1}
Let $C$ be a weakly toll convex set of a graph G. If $x$ is an
extreme vertex of $C$, then $x$ is a simplicial vertex in $G[C]$.
\end{lemma}

\begin{proof}
Suppose that $x$ is an extreme vertex of $C$ that is not
simplicial in $G[C]$. Then there exist two neighbors of $x$ in
$C$, say $u$ and $v$, which are not adjacent. But then $u, x, v$
is a weakly toll walk in $G$. Hence $C\setminus \{x\}$ is not a
weakly toll convex set, which is the desired contradiction.
\end{proof}

\begin{lemma}\label{l3}
A vertex $v$ of a proper interval graph $G$ is an extreme vertex
of the weakly toll convex set $V(G)$ if and only if $v$ is an end
simplicial vertex of $G$.
\end{lemma}

\begin{proof}
Suppose that $v$ is an extreme vertex that is not an end
simplicial vertex. Since $v$ is an extreme vertex, no induced path
contains $v$ as an internal vertex; hence, $v$ is a simplicial
vertex so that $v$ is not an end vertex. Using Corollary~\ref{t1}
and the observations after Theorem~\ref{tg}, $v$ is not an extreme
vertex, which is a contradiction.

Conversely, assume, in order to obtain a contradiction, that there
exists a weakly toll walk $\mathit{W}$ between two nonadjacent
vertices $x$ and $y$ of $G$ containing an end simplicial vertex
$v$ as an internal vertex. Write $\mathit{W}:
x,w_1,\ldots,w_i,v,w_{i+2},\ldots,w_n,y$. Since $v$ is an end
simplicial vertex, we can assume that there exists a proper
interval model $\{{\mathcal I}_u\}_{u \in V(G)}$ such that
${\mathcal I}_v$ appears as the first interval on the line. First,
we will show that $x$ is not adjacent to $v$. Suppose that $x$ is
adjacent to $v$. As $\mathit{W}$ is a weakly toll walk, $w_1=v$.
Moreover, since $v$ is a simplicial vertex, $w_i$ and $w_{i+2}$
are adjacent to $x$. Thus $w_i=w_{i+2}=v$. Clearly $w_k=v$ for $k
\in \{1,\ldots,n\}$. Then $\mathit{W}$ is the walk $x, v, y$.
Since $v$ is a simplicial vertex and $x, y$ are adjacent to $v$,
we have that $x$ is adjacent to $y$, which is a contradiction.

Since $x$ and $y$ are not adjacent to $v$, and $x$ is not adjacent
to $y$, we have ${\mathcal I}_v \cap {\mathcal I}_x=\emptyset$,
${\mathcal I}_x \cap {\mathcal I}_y=\emptyset$, and ${\mathcal
I}_v \cap {\mathcal I}_y=\emptyset$. Thus, we can assume that
${\mathcal I}_v$, ${\mathcal I}_x$, and ${\mathcal I}_y$ appear in
this order on the line.

On the other hand, the weakly toll walk $\mathit{W}:
x,w_1,\ldots,w_i,v,w_{i+2},\ldots,w_n,y$ goes from $x$ to $v$
through $w_1$, $w_1$ is the only vertex adjacent to $x$, and $G$
is a proper interval graph. Thus $L({\mathcal I}_x) \in {\mathcal
I}_{w_1}$ and $L({\mathcal I}_{w_1})<L({\mathcal I}_x)$. Also, the
weakly toll walk $\mathit{W}:
x,w_1,\ldots,w_i,v,w_{i+2},\ldots,w_n,y$ goes from $v$ to $y$
through $w_1$ because the only vertex adjacent to $x$ is $w_1$.
Thus $R({\mathcal I}_x) \in {\mathcal I}_{w_1}$ and $R({\mathcal
I}_x)<R({\mathcal I}_{w_1})$. Hence, we have ${\mathcal I}_x
\subsetneqq {\mathcal I}_{w_1}$, which is a contradiction because
$\{{\mathcal I}_u\}_{u\in V(G)}$ is a proper interval model of
$G$.
\end{proof}

\begin{lemma}\label{2}\hspace*{1cm}

\noindent $1.$ Let $x$, $y$ and $z$ be vertices forming an
asteroidal triple in a graph $G$. If $C$ is the weakly toll convex
hull of the set $\{x,y,z\}$, then $C$ does not have extreme
vertices.

\noindent $2.$ Let $a$, $b$, $c$, and $d$ be vertices inducing a
$K_{1,3}$ in a graph $G$. If $C$ is the weakly toll convex hull of
the set $\{a,b,c,d\}$, then $C$ does not have extreme vertices.

\end{lemma}

\begin{proof} \hspace*{1cm}

\begin{enumerate}

\item \label{item1} Let $w$ be a vertex of $C$, different from
$x$, $y$, and $z$. Assume, in order to obtain a contradiction,
that $w$ is an extreme vertex in $C$. Thus $C\setminus \{w\}$ is a
weakly toll convex set. Also $C\setminus \{w\}$ contains $x$, $y$,
and $z$. This is a contradiction.

Now, we will show that no vertex of the set $\{x,y,z\}$ is an
extreme vertex. Let $P$ and $Q$ be the induced paths between $x,y$
(avoiding neighbors of $z$), and $y,z$ (avoiding neighbors of
$x$), respectively. By concatenating $P$ and $Q$, we obtain a
weakly toll walk  between $x$ and $z$ containing $y$ (observe that
no vertex of $Q$ can be adjacent to $x$, and no vertex of $P$ can
be adjacent to $z$). Thus, there is a weakly toll walk between two
vertices of the asteroidal triple containing the other vertex of
the asteroidal triple. Hence no vertex forming the asteroidal
triple is an extreme vertex in $C$.

\item As in the proof of~\ref{item1}, if $w$ is a vertex of $C$,
different from $a$, $b$, $c$,  and $d$, then $w$ is not an extreme
vertex.

Let $a,b,c,d$ be the vertices of a $K_{1,3}$ such that $b$ has
degree three. Since $a,b,c,b,d$ is a weakly toll walk between $a$
and $d$, $b$ and $c$ are not extreme vertices of $K_{1,3}$.
Likewise, we prove that $a$ and $d$ are not extreme vertices of
$K_{1,3}$.

\end{enumerate}

\end{proof}

\section{Proper interval graphs as convex geometries}

Recall that if a graph is a convex geometry with respect to a
particular convexity, then $V(G)$ and every convex subset of $G$
is the convex hull of its extreme vertices with respect to that
convexity.

\begin{theorem}\label{t31}
A graph $G$ is a convex geometry with respect to the weakly toll
convexity if and only if $G$ is a proper interval graph.
\end{theorem}

\begin{proof}
Let $G$ be a convex geometry w.r.t. the weakly toll convexity. We
will show that $G$ is a proper interval graph, using induction on
the number of vertices of $G$. The claim is true if $G$ has one or
two vertices. Assume that $G$ has $n$ vertices and that the claim
is true for all graphs with fewer than $n$ vertices. Let $x$ be
any extreme vertex of $V(G)$. Clearly $G-x$ is a convex geometry.
By the induction hypothesis, $G-x$ is a proper interval graph. In
particular, $G-x$ is a chordal graph. By Lemma~\ref{1}, $x$ is a
simplicial vertex in $G$; thus, $G$ is chordal. If $G$ has an
asteroidal triple, then, by Lemma~\ref{2}, the weakly toll convex
hull $C$ of the asteroidal triple has no extreme vertices. This
implies that $C$ is not the weakly toll convex hull of its extreme
vertices, which is a contradiction. Thus $G$ is a chordal graph
with no asteroidal triple, and this implies that it is an interval
graph. Now, we claim that $G$ does not contain $K_{1,3}$ as an
induced subgraph. In order to obtain a contradiction, suppose that
$G$ contains $G'=K_{1,3}$ as an induced subgraph. Let $C$ be the
weakly toll convex hull of the set $V(G')$. It is clear, using
Lemma~\ref{2}, that $C$ has no extreme vertices. This implies that
$C$ is not the weakly toll convex hull of its extreme vertices,
which is a contradiction. Hence, $G$ is a proper interval graph.

Conversely, every convex subset of a proper interval graph $G$
induces a proper interval graph. Thus it suffices to show that
$V(G)$ is the convex hull of its extreme vertices. Since $G$ is a
proper interval graph, then it is an interval graph. By
Lemma~\ref{igw}, every vertex of $G$ that is not an end simplicial
vertex lies in a wealky toll walk between two end simplicial
vertices. Recall that, by Lemma~\ref{l3}, the set of end
simplicial vertices is equal to the set of extreme vertices in the
weakly toll convexity if $G$ is a proper interval graph. Hence
every vertex of $G$ that is not an extreme vertex lies in a weakly
toll walk between two extreme vertices. This implies that $G$ is a
convex geometry.
\end{proof}

\section{Some invariants associated with the weakly toll convexity}

In this section, we consider some standard invariants with respect
to the weakly toll convexity that have been extensively studied
for other graph convexities. We consider the weakly toll number
and the weakly toll hull number of a graph.

The definition of weakly toll interval for two vertices $u$ and
$v$ can be generalized to an arbitrary subset $S$ of $V(G)$ as
follows:
$$\mathit{WT}_G(S)=\bigcup_{u,v\in S} \mathit{WT}_G(u, v).$$
If $\mathit{WT}_G(S) = V(G)$, $S$ is called a \textit{weakly toll
set} of $G$. The order of a minimum weakly toll set in $G$ is
called the {\em weakly toll number} of $G$, and is denoted by
$\mathit{wtn}(G)$. For any non-trivial connected graph $G$, it is
clear that $2\leq \mathit{wtn}(G)\leq n$.

As mentioned earlier, the weakly toll convex hull of a set
$S\subseteq V(G)$ is defined as the intersection of all weakly
toll convex sets that contain $S$, and we will denote this set by
$[S]_{\mathit{WT}}$. A set $S$ is a \textit{weakly toll hull set}
of $G$ if $[S]_{{\mathit WT}}=V(G)$. The \textit{weakly toll hull
number} of $G$, denoted by $\mathit{wth}(G)$, is the minimum among
all the cardinalities of weakly toll hull sets.

Given a set $S \subseteq V(G)$, define $\mathit{WT}^k(S)$ as
follows: $\mathit{WT}^0(S) = S$ and $\mathit{WT}^{k}(S) =
\mathit{WT} (\mathit{WT}^{k-1}(S))$ for $k\geq 1$. Note that
$[S]_{{\mathit WT}} = \bigcup_{k\in \mathbb{N}} \mathit{WT}^k(S)$.
From the definitions, we immediately infer that every weakly toll
set is a weakly toll hull set, and hence $\mathit{wth}(G)\leq
\mathit{wtn}(G)$.

By the proof of Theorem~\ref{t31}, the weakly toll number (as well
as the weakly toll hull number) of a proper interval graph
coincides with the number of its extreme vertices. Indeed,
every weakly toll set of $G$ contains $\mathit{Ext}(G)$.
Furthermore, every vertex that is not an extreme vertex lies in a
weakly toll interval between two extreme vertices. We then derive
the following fact:

\begin{proposition}
If $G$ is a proper interval graph, then $\mathit{wtn}(G) =
\mathit{wth}(G) = |\mathit{Ext}(G)|$, where $\mathit{Ext}(G)$ is
the set of extreme vertices of $G$.
\end{proposition}

It is known that, among the trees, only paths are proper interval
graphs. In the following theorem, we determine the weakly toll
number and the weakly toll hull number of any tree.

\begin{theorem}
Let $G$ be a non-trivial tree. Then
$\mathit{wtn}(G)=\mathit{wth}(G)=2$.
\end{theorem}

\begin{proof}
Let $a$ and $b$ be two leaves of $G$. Clearly, every vertex of
$V(G)\setminus\{a,b\}$ lies in some weakly toll walk between $a$
and $b$; hence $\mathit{wtn}(G) \leq 2$. Since $G$ is a
non-trivial graph, $\mathit{wtn}(G)=2$. As $\mathit{wth}(G)\leq
\mathit{wtn}(G)$ and $G$ is a non-trivial graph, it follows that
$\mathit{wth}(G)=2$.
\end{proof}

We now need to recall a special representation of interval graphs.
More details can be found in~\citep*{fg}. Let $\mathcal{C}(G)$ be
the set of all maximal cliques of an interval graph $G$. A
\textit{canonical representation} $I$ of $G$ is a total order
$Q_1,\ldots,Q_k$ of the set $\mathcal{C}(G)$ in which for each
vertex $v$ of $G$, the cliques in the set $\mathcal{Q}_v =\{Q \in
\mathcal{C}(G) \mid v\in Q\}$ occur consecutively in the order
(see Figure~\ref{rc}). Maximal cliques $Q_1$ and $Q_k$ are called
\textit{end cliques} of the representation. In addition, for two
maximal cliques $Q_i,Q_j$ with $i \leq j$, we denote by
$G[Q_i,Q_j]$ the subgraph of $G$ induced by $Q_i\cup Q_{i+1}\cup
\ldots\cup Q_j$, and by $I[Q_i,Q_j]$ the canonical representation
$Q_i,Q_{i+1},..,Q_j$ of $G[Q_i,Q_j]$. The {\em clique intersection
graph} $K(G)$ is defined as follows:
$V(K(G))=\mathcal{C}(G)=\{Q_1,\ldots,Q_k\}$, and there is an edge
$Q_iQ_j$ in $K(G)$ if and only if $Q_i\cap Q_j\neq\emptyset$. We
denote by $G_I$ the acyclic spanning subgraph of $K(G)$ consisting
of the path $Q_1,Q_2,\ldots,Q_k$. The graph $G_I$ is intrinsically
associated with the canonical representation $I$. It is a
particular case of a {\em clique tree} of a chordal graph (recall
that interval graphs are chordal). For more details,
see~\citep*{bu,ga}. In the case of interval graphs, clique trees
are paths. Note that in every canonical representation $I$ of
$star_{1,2,2}$, $B_n \; (n>5)$, and the bull graph, $N[v]$ is a
clique that is not an end clique (see Figure~\ref{grafica2}).
Moreover, if $G$ is an interval graph which contains
$star_{1,2,2}$, $B_n \; (n>5)$, or the bull graph as an induced
subgraph and $Q \in \mathcal{C}(G)$ contains $N[v]$, then in each
canonical representation $I$ of $G$, $Q$ is an internal vertex of
$G_I$.

In this context, a given simplicial vertex $v$ in an interval
graph $G$ is an end vertex if there exists some canonical
representation where $N[v]$ is an end clique of this
representation. This can be proved as follows: if $v$ is a
simplicial vertex and $N[v]$ is an end clique of a canonical
representation, then $v$ cannot be the designated vertex in any
configuration depicted in Figure~\ref{grafica2}. Thus, by Theorem
2.2, $v$ is an end vertex.

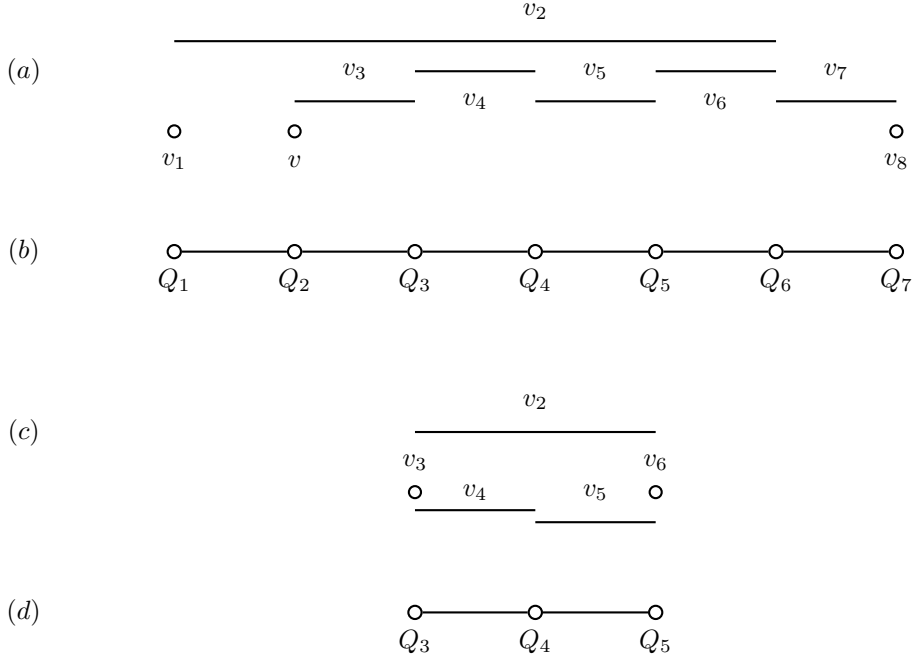
\begin{figure}[ht]
\begin{center}

\begin{tikzpicture}[scale=0.8]

\pgfsetlinewidth{0.8pt}

\tikzset{vertex/.style={circle,  draw, minimum size=5pt, inner
sep=0pt}}

\node [vertex] (Q1) at (-10, -1) [label=below:$Q_1$]{}; \node
[vertex] (Q2) at (-8, -1) [label=below:$Q_2$]{} edge (Q1); \node
[vertex] (Q3) at (-6, -1) [label=below:$Q_3$]{} edge (Q2); \node
[vertex] (Q4) at (-4, -1) [label=below:$Q_4$]{} edge (Q3); \node
[vertex] (Q5) at (-2, -1) [label=below:$Q_5$]{} edge (Q4); \node
[vertex] (Q6) at (0, -1) [label=below:$Q_6$]{} edge (Q5); \node
[vertex] (Q7) at (2, -1) [label=below:$Q_7$]{} edge (Q6);

 \draw (-10,1) circle (1mm);
 \node (h) at (-10,0.5) {$v_1$};
 \node (i) at (-8,0.5) {$v$};
 \node (j) at (2,0.5) {$v_8$};
 \draw (-8,1) circle (1mm);
 \draw (2,1) circle (1mm);
 \draw (-8,1.5)--(-6,1.5);
 \node (k) at (-7,2) {$v_3$};
 \draw (-6,2)--(-4,2);
 \node (l) at (-5,1.5) {$v_4$};
 \draw (-4,1.5)--(-2,1.5);
 \node (m) at (-3,2) {$v_5$};
 \draw (-2,2)--(0,2);
 \node (n) at (-1,1.5) {$v_6$};
 \draw (0,1.5)--(2,1.5);
 \node (o) at (1,2) {$v_7$};
 \draw (-10,2.5)--(0,2.5);
 \node (p) at (-4,3) {$v_2$};

\draw (-6,-5) circle (1mm); \node (w) at (-6,-4.5){$v_3$}; \draw
(-2,-5) circle (1mm); \node (z) at (-2,-4.5) {$v_6$}; \draw
(-6,-4)--(-2,-4); \node (1) at (-4, -3.5) {$v_2$}; \draw
(-6,-5.3)--(-4,-5.3); \node (2) at (-5,-5) {$v_4$}; \node (3) at
(-3,-5) {$v_5$}; \draw (-4,-5.5)--(-2,-5.5);

\node (q) at (-12.5,2) {$(a)$}; \node (q) at (-12.5,-1) {$(b)$};
\node (q) at (-12.5,-4) {$(c)$}; \node (w) at (-12.5,-7) {$(d)$};

\node [vertex] (Q3) at (-6, -7) [label=below:$Q_3$]{}; \node
[vertex] (Q4) at (-4, -7) [label=below:$Q_4$]{} edge (Q3); \node
[vertex] (Q5) at (-2, -7) [label=below:$Q_5$]{} edge (Q4);

\end{tikzpicture}

\end{center}
\caption{(a) Interval model of graph $G=B_8$ (see Figure 1); (b)
clique tree $G_I$, associated with canonical representation
$I=Q_1,Q_2,\ldots,Q_7$; (c) interval model of the graph
$H=G[Q_3,Q_5]$; (d) clique tree $H_{J}$, associated with canonical
representation $J=I[Q_3,Q_5]$.} \label{rc}

\end{figure}

In what follows, we show some properties of the canonical
representation $I$ of an interval graph $G$.

\begin{proposition} \label{p3}
Let $G$ be a connected interval graph, and $I$ a canonical
representation of $G$ such that $Q$ is an end vertex of $G_I$ and
$Q',Q''$ is a pair of vertices of $G_I$ such that $Q \cap Q' \neq
\emptyset$ and $Q \cap Q'' \neq \emptyset$. If $Q, Q', Q''$ appear
in this order (or in reverse order) in $I$, then $Q \cap Q''
\subseteq Q \cap Q'$.
\end{proposition}

\begin{proof}
Assume $Q, Q', Q''$ appear in this order in $I$. Let $x$ be a
vertex in $Q \cap Q''$. Clearly $Q, Q'' \in \mathcal{Q}_x$. Since
$Q' \in I[Q, Q'']$, it follows that $Q' \in \mathcal{Q}_x$. Thus
$x \in Q \cap Q'$, that is, $Q \cap Q''\subseteq Q \cap Q'$.

If the vertices appear in the order $Q'', Q', Q$ in $I$, then $Q'
\in I[Q'', Q]$, and the proof uses the same arguments.
\end{proof}

\begin{proposition}\label{p4}
Let $G$ be a connected interval graph, $I$ a canonical
representation of $G$, and $s_i$ an end simplicial vertex of $G$
such that $N[s_i]$ is end clique of $I$ for $i=1,2$. If there
exist maximal cliques $Q_1, Q_2, Q_3$, and $Q_4$ of $G$ such that
$N[s_1], Q_1, Q_2, Q_3, Q_4, N[s_2]$ appear in this order or in
reverse order in $I$, $Q_1 \cap Q_2 \subseteq N[s_2]$ and $Q_3\cap
Q_4 \subseteq N[s_1]$, then $Q_3 \cap Q_4=Q_1 \cap Q_2$.
\end{proposition}


\begin{proof}
Let $x\in Q_1\cap Q_2$. Since $Q_1 \cap Q_2 \subseteq N[s_2]$, it
follows that $\{Q_1,Q_2,N[s_2]\}\subseteq \mathcal{Q}_x$. Recall
that the cliques in $\mathcal{Q}_x$ occur consecutively in $I$.
This implies that $Q_3\in \mathcal{Q}_x$ and $Q_4\in
\mathcal{Q}_x$, that is, $x\in Q_3\cap Q_4$. Hence, $Q_1\cap
Q_2\subseteq Q_3\cap Q_4$. Analogously, we can prove that $Q_3
\cap Q_4 \subseteq Q_1 \cap Q_2$. Therefore $Q_3 \cap Q_4 = Q_1
\cap Q_2$.
\end{proof}

\begin{proposition}\label{p2}
Let $G$ be a connected interval graph, $I$ a canonical
representation of $G$, and $s$ an end simplicial vertex of $G$
such that $N[s]$ and $Q$ are the end cliques of $I$. If $G-N[s]$
has $k+1$, $k\geq 1$, connected components, then there exist $k$
vertices $q_1, \ldots, q_k$ in $G-N[s]$ such that:

\begin{enumerate}

\item every $q_i$ is an end simplicial vertex of $G$;

\item $N[s], N[q_1], \ldots, N[q_k]$ appear in this order or in
reverse order in $I$;

\item if $k\geq 2$, then for $i=2,\ldots,k$, $N[q_{i}] \cap N[s]
\subseteq N[q_{i-1}] \cap N[s]$;

\item for $i=1,\ldots,k$, there exists a maximal clique $Q'_i$
such that $N[q_i] \cap Q'_i \subseteq N[s]$, and $N[q_i], Q'_i$
are consecutive vertices in $G_I$;

\item for $i=1,\ldots,k$, there exists another canonical
representation $I_i$ such that $N[q_i]$ and $Q$ are its end
cliques;

\item $G[N[q_k],Q]-N[q_k]$ is a connected graph.

\end{enumerate}

\end{proposition}

\begin{proof} Since $G-N[s]$ is not a connected graph, there exist consecutive vertices $Q_i,Q'_i$ in $G_I$ for $i=1, \ldots, k$ such that $Q_1 \neq N[s]$, $Q_i\cap Q'_i \subseteq N[s]$, and the vertices of $Q_i\setminus N[s]$ and $Q'_i\setminus N[s]$ belong to different connected components of $G-N[s]$, for $i=1,\ldots,k$. Note that, for $i = 1, \ldots, k-1$, clique $Q_{i+1}$ can coincide with $Q'_i$; in addition, if $k=1$, the graph $G-N[s]$ has exactly two connected components.

In order to fix ideas, assume that
$N[s],Q_1,Q'_1,\ldots,Q_k,Q'_k,Q$ appear in this order in $I$. The
graph $G_I-Q_iQ'_i+Q'_iN[s]$ is clearly associated with another
canonical representation $I_i$ of $G$, for $i=1,\ldots,k$. Note
that $Q_i$ is an end clique of $I_i$, so there exist simplicial
end vertices $q_i \in Q_i$, which are end vertices of $G$, for
$i=1,\ldots,k$. As $q_i$ is simplicial, $N[q_i]$ is a maximal
clique of $G$ and $Q_i=N[q_i]$ for $i=1,\ldots,k$. Furthermore, by
Proposition~\ref{p3}, if $k\geq 2$, then $N[q_{i}] \cap N[s]
\subseteq N[q_{i-1}] \cap N[s]$, for $i=2,\ldots,k$.

Let $G_2=G[N[q_k],Q]$. Since $G-N[s]$ has $k+1$ connected
components, it follows that $N[q_k], Q'_k$ are the only cliques of
$G_2$ such that $N[q_k] \cap Q'_k \subseteq N[s]$. Thus
$G_2-N[q_k]$ is a connected graph. Note that $G_2-N[s]$ has
exactly two connected components.

\end{proof}

Two vertices $u$ and $v$ of a graph $G$ are called \textit{twins}
if $N[u]=N[v]$.

\begin{proposition}\label{p1}
Let $G$ be a connected interval graph, $I$ a canonical
representation of $G$, and $s_1$ and $s_2$ end vertices of $G$
such that $N[s_1]$ and $N[s_2]$ are distinct end cliques of $I$.
Then, every vertex of $G-(N[s_1]\cup N[s_2])$ lies in a weakly
toll walk between $s_1$ and $s_2$.
\end{proposition}

\begin{proof} Let $y$ and $w$ be vertices of $G$ such that $y \in N(s_1)$, $w \in N(s_2)$, and $|\mathcal{Q}_v|$ is maximum for each $v \in \{y,w\}$. Note that $y$ may be equal to $w$. Let $x$ be a vertex of $G-(N[s_1]\cup N[s_2])$.

If $y=w$ or $y,w$ are twins then $\mathcal{Q}_y=\mathcal{C}(G)$.
Therefore, $W:s_1,y,x,y,s_2$ is a weakly toll walk between $s_1$
and $s_2$, which captures $x$.

Assume that $y \neq w$ and $y,w$ are not twins. If $y$ is adjacent
to $w$, then there is a maximal clique $Q\in \mathcal{C}(G)$
containing both $y$ and $w$. Thus, $\{N[s_1],Q\}\subseteq
\mathcal{Q}_y$ and $\{Q,N[s_2]\}\subseteq \mathcal{Q}_w$, and the
cliques $N[s_1],Q,N[s_2]$ appear in this order (or in reverse
order) in $I$. Suppose without loss os generality that
$N[s_1],Q,N[s_2]$ appear in this order in $I$. Since the cliques
in $\mathcal{Q}_y$ (or in $\mathcal{Q}_w$) occur consecutively in
$I$, $\mathcal{Q}_y$ contains all the cliques in $I$ from $N[s_1]$
to $Q$, and $\mathcal{Q}_w$ contains all the cliques in $I$ from
$Q$ to $N[s_2]$. Therefore, $\mathcal{Q}_y \cup
\mathcal{Q}_w=\mathcal{C}(G)$. Moreover, $x$ is clearly adjacent
to $y$ or $w$. Thus either $W_1:s_1,y,x,y,w,s_2$, \
$W_2:s_1,y,w,x,w,s_2$, or $W_3:s_1,y,x,w,s_2$ is a weakly toll
walk between $s_1$ and $s_2$, which captures $x$.

Now, suppose that $y$ is not adjacent to $w$. Let
$P:y,x_1,..,x_n,w$ be an induced path between $y$ and $w$ in $G$.
By the choice of $y$ and $w$, $x_i \notin N[s_1] \cup N[s_2]$.

Let $x_0=y$ and $x_{n+1}=w$. Since $x_ix_{i+1}\in E(G)$, let $Q_i$
be a clique of $\mathcal{C}(G)$ containing both $x_i$ and
$x_{i+1}$, for $i=0,\ldots,n$. Since $P$ is induced,
$Q_i\cap\{x_0,\ldots,x_{n+1}\}=\{x_i,x_{i+1}\}$, and this implies
that (a) $Q_i\neq Q_j$, for distinct indices
$i,j\in\{0,\ldots,n\}$, and (b) $Q_0,Q_1,\ldots,Q_n$ or
$Q_n,Q_{n-1},\ldots,Q_0$ is a total order. Hence,
$N[s_1],Q_0,\ldots,Q_n,N[s_2]$ appear in this order (or in reverse
order) in $I$. Suppose without loss of generality that
$N[s_1],Q_0,\ldots,Q_n,N[s_2]$ appear in this order in $I$. Note
that (i) $\mathcal{Q}_y$ contains all the cliques in $I$ from
$N[s_1]$ to $Q_0$, (ii) $\mathcal{Q}_{x_i}$ contains all the
cliques in $I$ from $Q_i$ to $Q_{i+1}$, for $i=0,\ldots,n$, and
(iii) $\mathcal{Q}_w$ contains all the cliques in $I$ from $Q_n$
to $N[s_2]$. Therefore, $\mathcal{Q}_y \cup \mathcal{Q}_w \cup
\bigcup_{i=1}^n \mathcal{Q}_{x_i}=\mathcal{C}(G)$.

To conclude the proof, observe that $x=x_i$ or $x$ must be
adjacent to $y$ or $w$ or $x_i$ for some $1 \leq i \leq n$.
Therefore, $s_1,P,s_2$ or $s_1,y,x,P,s_2$ or $s_1,P,x,w,s_2$ or
$s_1,y,x_1,\ldots,x_i,x,x_i,\ldots,x_n,w,s_2$ (for some $i \in
\{1,\ldots,n\}$) is a weakly toll walk between $s_1$ and $s_2$,
which captures $x$.

\end{proof}

\begin{proposition}\label{l1}
Let $G$ be a connected interval graph, let $s_1 \neq s_2$ be two
non-twin simplicial vertices of $G$ such that $G-N[s_1]$ and
$G-N[s_2]$ are connected graphs, and let $S_i=\{s_i\} \cup \{x \in
V(G) : x \ \text{is a twin of} \ s_i\}$ for $i=1,2$. Then,

\begin{enumerate}

\item[$(1)$] $S_1\cup S_2$ is a weakly toll set of $G$;

\item[$(2)$] if there exists $v \notin N[s_1]$ such that for some
$y \in N[s_1]$, $\mathcal{Q}_v \varsubsetneq \mathcal{Q}_y$, and
there exists an induced path $P$ between $y$ and $s_2$ with $P
\cap N[v] =\{y\}$, then $S=S_1 \cup \{s_2,v\}$ is a weakly toll
set of $G$;

\item[$(3)$] if there exist $v \notin N[s_1]$ and $w \notin
N[s_2]$ such that for some $y \in N[s_1]$ and $z \in N[s_2]$,
$\mathcal{Q}_v \varsubsetneq \mathcal{Q}_y$ and $\mathcal{Q}_w
\varsubsetneq \mathcal{Q}_z$, and there exist induced paths $P$
between $y$ and $s_2$ with $P \cap N[v] =\{y\}$ and $P_1$ between
$z$ and $s_1$ with $P_1 \cap N[w] =\{z\}$, then $S=\{s_1, v, s_2,
w\}$ is a weakly toll set of $G$;

\item[$(4)$] $2 \leq \mathit{wtn}(G) \leq |S_1 \cup S_2|$.

\end{enumerate}

\end{proposition}

\begin{proof}\hspace*{1cm}

If $G-N[s_1]$ and $G-N[s_2]$ are connected graphs then, in every
canonical representation of $G$, $N[s_1]$ and $N[s_2]$ are end
cliques.

\begin{enumerate}

\item[$(1)$]

Let $x\in V(G)\setminus (N[s_1]\cup N[s_2])$. By Proposition
\ref{p1}, $x$ lies in a weakly toll walk between $s_1$ and $s_2$.
Now, for $i \in \{1,2\}$, let $x \in N[s_i] \setminus S_i$. Since
$x \notin S_i$ and $G-N[s_i]$ is a connected graph, there exists
at least a vertex $y \in N(x)\setminus N[s_i]$ and an induced path
$P: y, \ldots, s_j$ avoiding the neighbors of $s_i$. Note that no
vertex of $P$ is a neighbor of $s_i$. Then $s_i,x,y,\ldots,s_j$ is
a walk between $s_i$ and $s_j$ which contains $x$ as internal
vertex; in addition, this walk contains an induced path $P_1$
between $s_i$ and $s_j$ which contains $x$ as an internal vertex.
Hence $P_1$ is a weakly toll walk between $s_i$ and $s_j$
containing $x$ as an internal vertex. Thus $S_1 \cup S_2$ is a
weakly toll set of $G$ and $\mathit{wtn}(G)\leq |S_1 \cup S_2|$.

\item[$(2)$]

Let $P=y,y_1,\ldots,y_n,s_2$ such that $P \cap N[v]=\{y\}$. Since
$s_1 \neq s_2$, $v \notin N[s_1]$, and $\mathcal{Q}_v$ is properly
contained in $\mathcal{Q}_y$, we have that $y_1 \notin N[s_1]$.
Let $s$ be a twin of $s_2$. Since $y_i \notin N[v]$, the walk
$s_1,y,y_1,\ldots,y_n,s,y_n,y_{n-1},\ldots,y_1,y,v$ is a weakly
toll walk between $s_1$ and $v$, which captures $s$. Hence $S=S_1
\cup \{s_2,v\}$ is a weakly toll set and $\mathit{wtn}(G)\leq
|S|$.

\item[$(3)$]

Every vertex of $V(G)\setminus (S_1 \cup S_2)$ is contained in the
weakly toll interval of $\{s_1,s_2\}$, every vertex of $S_2$ is
contained in the weakly toll interval of $\{s_1,v\}$, and every
vertex of $S_1$ is contained in the weakly toll interval of
$\{s_2,w\}$. Hence $S=\{s_1, v, s_2, w\}$ is a weakly toll set and
$\mathit{wtn}(G)\leq |S|$.

\item[$(4)$] From the above, $2 \leq \mathit{wtn}(G) \leq |S_1
\cup S_2|$.

\end{enumerate}
\end{proof}

By the above proposition, the weakly toll number of an interval
graph having two simplicial vertices $s_1$ and $s_2$, which are
always in end cliques of every representation, is a number between
two and the number of twins of $s_1$ and $s_2$ plus two. The
following example is elucidative.

\begin{figure}[ht]
\begin{center}
\begin{tikzpicture}[scale=0.8]
\node[draw,circle] (2) at (-1.5,0) {$1$}; \node[draw,circle] (3)
at (0,0) {$3$}; \node[draw,circle] (7) at (1.5,0) {$4$};
\node[draw,circle] (9) at (-0.75,1) {$2$}; \draw
(3)--(2)--(9)--(3)--(7);
\end{tikzpicture}
\end{center}
\caption{Interval graph with weakly toll number three. \label{5}}
\end{figure}
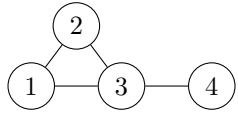

In Figure~\ref{5}, the graph has weakly toll number three. Observe
that there is a weakly toll walk $W$ between $1$ and $4$, namely
$W:1,3,4$, which captures vertex $3$, but there is no weakly toll
walk between $1$ and $4$ which captures vertex 2. Hence,
$\{1,2,4\}$ is a weakly toll set of $G$.

In the following theorem, we show that an interval graph with no
end simplicial vertex having a twin has weakly toll number at most
three and weakly toll hull number equal to two.

\begin{theorem}\label{teo} Let $G$ be a connected interval graph. Let $s_1$ be
an end vertex of $G$ such that $N[s_1]$ is an end clique of a
canonical representation $I$ of $G$. If $G-N[s_1]$ is not a
connected graph, then $\mathit{wtn}(G)\leq 3$ and
$\mathit{wth}(G)=2$.
\end{theorem}

\begin{proof}First, we will prove that $wtn(G)\leq 3$.

Let $s_2$ be an end vertex of $G$ such that $N[s_2]$ is the other
end clique of the canonical representation $I$.

By Proposition \ref{p1}, every vertex $x \notin N[s_1] \cup
N[s_2]$ lies in a weakly toll walk between $s_1$ and $s_2$.

On the other hand, if $x \in N[s_1] \cap N[s_2]$ then $x$ lies in
the weakly toll walk $W=s_1,x,s_2$, between $s_1$ and $s_2$.

Thus we just need to study vertices $x \in (N[s_1] \cup N[s_2])\setminus
(N[s_1] \cap N[s_2])$.

Assume that $G-N[s_1]$ has at least two connected components. By
Proposition \ref{p2}, if $G-N[s_1]$ has $k+1$ connected components,
there exist end simplicial vertices $q_1,\ldots, q_k$ of $G$ and $Q'_1,\ldots,Q'_k$ maximal cliques of $G$ such that $N[q_i],Q'_i$ are consecutive vertices in $I$, $N[q_{i+1}] \cap N[s_1] \subseteq N[q_i] \cap
N[s_1]$, and $N[q_i] \cap Q'_i \subseteq N[s_1]$ for
$i=1,\ldots,k-1$.

We will show that $S=\{s_1,q_1,s_2\}$ is a weakly toll set of $G$.

\begin{figure}[ht!]

\centering

\begin{tikzpicture}[scale=1]

\pgfsetlinewidth{1pt}

\tikzset{vertex/.style={circle,  draw, minimum size=5pt, inner
sep=1.5pt}}

\node (13) at (-2,0) {$G$};

\node [vertex] (a) at (0, 0) [label=left:$a$]{}; \node [vertex]
(b) at (1, 1) [label=above:$b$]{}; \node [vertex] (c) at (3, -1)
[label=below:$q_2$]{}; \node [vertex] (d) at (2, 0)
[label=below:$d$]{}; \node [vertex] (e) at (4, 0)
[label=below:$e$]{}; \node [vertex] (q1) at (2, 1)
[label=above:$q_1$]{}; \node [vertex] (s1) at (1, -1)
[label=below:$s_1$]{}; \node [vertex] (s2) at (5, -1)
[label=below:$s_2$]{};

\draw [] (a) to (b); \draw [] (a) to (d); \draw [] (a) to (s1);
\draw [] (b) to (q1); \draw [] (b) to (d); \draw [] (c) to (d);
\draw [] (d) to (e); \draw [] (d) to (q1); \draw [] (d) to (s1);
\draw [] (e) to (s2);

\node [vertex] (s1) at (0, -3) [label=above:$s_1$]{}; \node
[vertex] (q1) at (4, -3) [label=above:$q_1$]{}; \node [vertex]
(q2) at (6, -3) [label=above:$q_2$]{}; \node [vertex] (s2) at (10,
-3) [label=above:$s_2$]{}; \node (A) at (3,-3.8) {$d$}; \draw
(0,-4)--(8,-4); \node (B) at (1,-4.3) {$a$}; \draw
(0,-4.5)--(2,-4.5); \node (C) at (3,-4.8) {$b$}; \draw
(2,-5)--(4,-5); \node (D) at (9,-4.3) {$e$}; \draw
(8,-4.5)--(10,-4.5);

\node (13) at (-2,-7) {$G_I$};

\node [vertex] (1) at (0, -7) {$a,d,s_1$}; \node [vertex] (2) at
(2, -7) {$a,b,d$}; \node [vertex] (3) at (4, -7) {$b,d,q_1$};
\node [vertex] (4) at (6, -7) {$d,q_2$}; \node [vertex] (5) at (8,
-7) {$d,e$}; \node [vertex] (6) at (10, -7) {$e,s_2$};

\draw [] (1) to (2); \draw [] (2) to (3); \draw [] (3) to (4);
\draw [] (4) to (5); \draw [] (5) to (6);

\end{tikzpicture}
\caption{The graph $G-N[s_1]$ is not connected, but $G-N[s_2]$ is
connected. Note that $G_2=G[N[q_2], N[s_2]]=G[\{q_2,d,e,s_2\}]$
and $Q'_2=\{d,e\}$. By (1) in Proposition~\ref{l1}, $q_2,d,e,s_2$
is a weakly toll walk, and since $N[q_2]\cap Q'_2\subseteq
N[s_1]$, it follows that $d \in N[s_1]$ and $s_1,d,q_2,d,e,s_2$ is
a weakly toll walk which captures $d$, $e$ and $q_2$. In addition,
$G_1=G[N[s_1],N[q_1]]=G[\{s_1,a,b,d,q_1\}]$. Since $G_1-N[q_1]$ is
a connected graph, by (1) in Proposition ~\ref{l1}, $s_1,a,b,q_1$
is a weakly toll walk which captures $a$ and
$b$.\label{fig:case1}}

\end{figure}
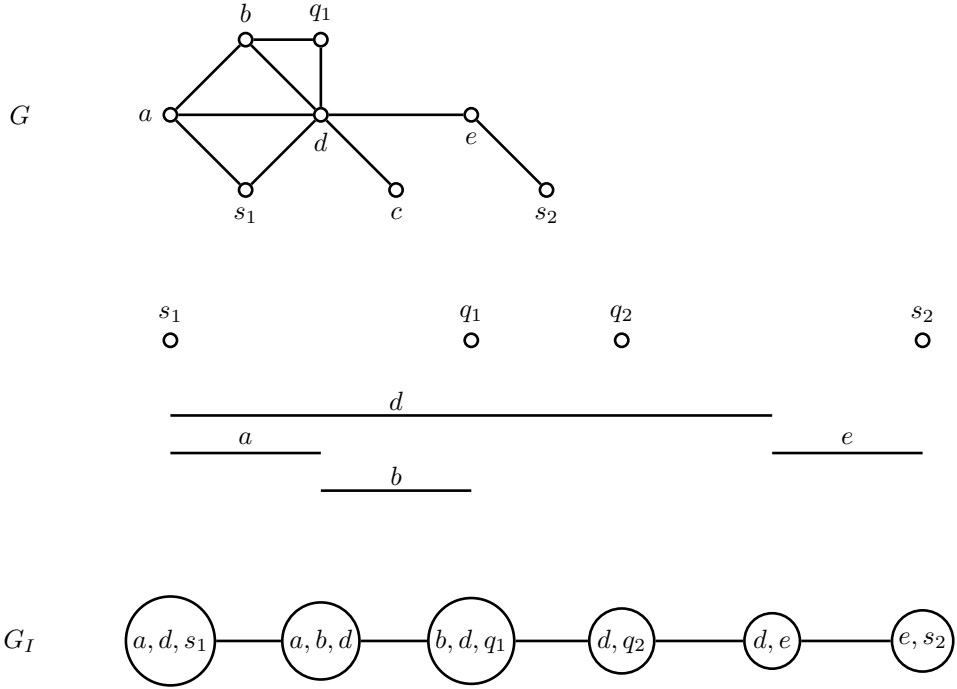

\begin{figure}[ht]

\centering{}

\begin{tikzpicture}[scale=1]

\pgfsetlinewidth{1pt}

\tikzset{vertex/.style={circle,  draw, minimum size=5pt, inner
sep=1.5pt}}

\node (13) at (-2,0) {$G$};

\node [vertex] (a) at (0, 0) [label=left:$a$]{}; \node [vertex]
(b) at (1, 1) [label=above:$b$]{}; \node [vertex] (c) at (3, -1)
[label=below:$c$]{}; \node [vertex] (d) at (2, 0)
[label=below:$d$]{}; \node [vertex] (e) at (4, 0)
[label=below:$e$]{}; \node [vertex] (f) at (6, 0)
[label=right:$f$]{}; \node [vertex] (q1) at (2, 1)
[label=above:$q_1$]{}; \node [vertex] (s1) at (1, -1)
[label=below:$s_1$]{}; \node [vertex] (q1') at (5, 1)
[label=above:$q'_1$]{}; \node [vertex] (s2) at (5, -1)
[label=below:$s_2$]{};

\draw [] (a) to (b); \draw [] (a) to (d); \draw [] (a) to (s1);
\draw [] (b) to (q1); \draw [] (b) to (d); \draw [] (c) to (d);
\draw [] (d) to (e); \draw [] (d) to (q1); \draw [] (d) to (s1);
\draw [] (e) to (f); \draw [] (e) to (q1'); \draw [] (e) to (s2);
\draw [] (f) to (q1'); \draw [] (f) to (s2);

\node [vertex] (s1) at (0, -3) [label=above:$s_1$]{}; \node
[vertex] (q1) at (4, -3) [label=above:$q_1$]{}; \node [vertex] (c)
at (6, -3) [label=above:$c$]{}; \node [vertex] (q'1) at (10, -3)
[label=above:$q'_1$]{}; \node [vertex] (s2) at (12, -3)
[label=above:$s_2$]{};

\node (A) at (3,-3.8) {$d$}; \draw (0,-4)--(8,-4); \node (B) at
(1,-4.3) {$a$}; \draw (0,-4.5)--(2,-4.5); \node (C) at (3,-4.8)
{$b$}; \draw (2,-5)--(4,-5); \node (D) at (10,-4.3) {$e$}; \draw
(8,-4.5)--(12,-4.5); \node (F) at (11,-5.1) {$f$}; \draw
(10,-5.4)--(12,-5.4);

\node (13) at (-2,-7) {$G_I$};

\node [vertex] (1) at (0, -7) {$a,d,s_1$}; \node [vertex] (2) at
(2, -7) {$a,b,d$}; \node [vertex] (3) at (4, -7) {$b,d,q_1$};
\node [vertex] (4) at (6, -7) {$d,c$}; \node [vertex] (5) at (8,
-7) {$d,e$}; \node [vertex] (6) at (10, -7) {$e,f,q'_1$}; \node
[vertex] (7) at (12, -7) {$e,f,s_2$};

\draw [] (1) to (2); \draw [] (2) to (3); \draw [] (3) to (4);
\draw [] (4) to (5); \draw [] (5) to (6); \draw [] (6) to (7);

\end{tikzpicture}

\caption{The walk $s_1,d,b,q_1,d,c,d,e,q'_1,e,s_2$ is a weakly
toll walk which captures $V(G)\setminus \{a,f\}$. The walk
$s_1,a,b,q_1$ is a weakly toll walk which captures $a$ and $b$.
Finally, the walk $s_1,d,e,f,e,d,q_1$ is a weakly toll walk which
captures $f$.\label{fig:case2}}

\end{figure}
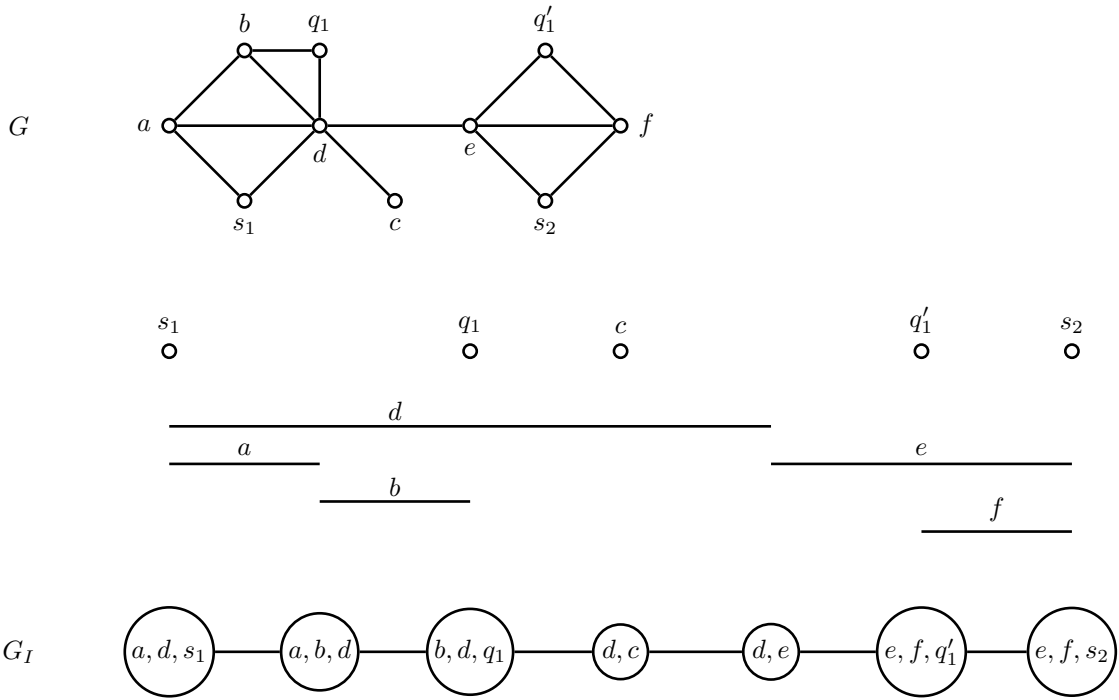

Let $w_1$ and $w_n$ be vertices in $N(s_1)$ and $N(s_2)$,
respectively, such that $|\mathcal{Q}_{w_1}|$ and
$|\mathcal{Q}_{w_n}|$ are maximum. Clearly, $q_1 \in N(w_1)$.

First, suppose that $x \in N(s_1)$. If $x \in N(q_1)$ then
$s_1,x,q_1$ is a weakly toll walk.

Assume that $x \notin N(q_1)$.  If $w_1$ and $w_n$ are twins, then
$s_2,w_1,x,w_1,q_1$ is a weakly toll walk.

Assume now that $w_1$ and $w_n$ are not twins. Note that $q_1 \notin
N(w_n)$. If $\mathcal{Q}_{w_1} \cap \mathcal{Q}_{w_n} \neq
\emptyset$, then the walk $W = s_2,w_n,w_1,x,w_1,q_1$ is a weakly toll walk.

Suppose that $\mathcal{Q}_{w_1} \cap \mathcal{Q}_{w_n} =
\emptyset$. Thus, there exists a connected component $C$ of $G\setminus
(N[s_1] \cup N[s_2])$ whose vertices are in $G[Q'_k,Q]$, where
$Q=N[s_2]$ if $G-N[s_2]$ is a connected graph, or, by Proposition
\ref{p2}, $Q \neq N[s_2]$ if $G-N[s_2]$ is not a
connected graph; in addition, there exist vertices
$a \in N(w_1) \cap V(C)$ and $b \in N(w_n) \cap V(C)$.
Let $a,y_1, \ldots, y_m, b$ be an induced path
between $a$ and $b$. Let $i$ and $j$ be the first and the last
indices such that $y_i$ is adjacent to $w_n$ and $y_j$ is adjacent
to $w_1$. Since $w_1$ is not adjacent to $w_n$, $j\leq i$. Thus, the walk
$q_1, w_1,x,w_1,y_j, y_{j+1}, \ldots, y_i,w_n,s_2$ is a weakly
toll walk.

Now, suppose that $x \in N(s_2)$.
If $w_1$ and $w_n$ are twins, then
$s_1,w_n,x,w_n,q_1$ is a weakly toll walk.

Consider now the case in which $w_1$ and $w_n$ are not twins.
In this situation, $q_1 \notin N(w_n)$.
If $\mathcal{Q}_{w_1} \cap \mathcal{Q}_{w_n} \neq
\emptyset$ then $s_1,w_1,w_n,x,w_n,w_1,q_1$ is a weakly toll walk.

Suppose now that $\mathcal{Q}_{w_1} \cap \mathcal{Q}_{w_n} =
\emptyset$. Thus, as exposed above, there exists a connected
component $C$ of $G\setminus (N[s_1]\cup N[s_2])$ whose vertices are
in $G[Q'_k,Q]$; in addition, there exist vertices
$a \in N(w_1) \cap V(C)$ and $b \in N(w_n) \cap V(C)$.
Let $a,y'_1, \ldots,y'_p, b$ an induced path
between $a$ and $b$. Let $i$ and $j$ be the first and the last
indices such that $y'_i$ is adjacent to $w_n$ and $y'_j$ is adjacent
to $w_1$. Since $w_1$ is not adjacent to $w_n$, $j\leq i$. Thus, the walk
$q_1,y'_j, \ldots, y'_i, w_n,x,w_n,y'_i, \ldots, y'_j,w_1,s_1$ is
a weakly toll walk.

From the above, $S=\{s_1,s_2,q_1\}$ is a weakly toll set of $G$.
Figures ~\ref{fig:case1}, \ref{fig:case2} illustrate the idea of
the proof. Therefore, $wtn(G)\leq 3$.

\medskip

In order to prove that $wth(G)=2$, first observe that, by
Proposition \ref{p1}, every end simplicial vertex, distinct from
$s_1$ and $s_2$, lies in a weakly toll walk between $s_1$ and
$s_2$.

Let $x$ be a vertex of $G$ that is not an end simplicial vertex.
Since $G$ is an interval graph, every vertex $x$ of $G$ that is
not an end simplicial vertex lies in a tolled walk between two end
simplicial vertices, and every end simplicial vertex lies in a
weakly toll walk between $s_1$ and $s_2$; thus, it is possible to
build a weakly toll walk between $s_1$ and $s_2$ which captures
$x$ using vertices of both walks.

Therefore, $S=\{s_1,s_2\}$ is a weakly toll hull set of $G$, i.e.,
$wth(G)=2$.
\end{proof}

In the remainder of this section, we consider an interval graph
$G$, and study the weakly toll number of $G$. Let $s_1$ and $s_2$
be two end simplicial vertices of $G$ such that $N[s_1]$ and
$N[s_2]$ are the end cliques of a canonical representation $I$ of
$G$.

If $G-N[s_i]$ is a connected graph for $i=1,2$, then, by
Proposition~\ref{l1}, it follows that
$$\mathit{wtn}(G)\leq |S_1\cup S_2|.$$


Suppose that $G-N[s_i]$ is not a connected graph for some $i \in
\{1,2\}$. By Theorem \ref{teo}, $S=\{s_1,q_1,s_2\}$ is a weakly
toll set of $G$.

Hence $wtn(G)\leq 3$.

From the above, we obtain the following result.

\begin{corollary}
Let $G$ be an interval graph, $s_1,s_2$ be two end simplicial
vertices of $G$, which are not twins, such that $N[s_1]$ and
$N[s_2]$ are end cliques of a canonical representation $I$ of $G$,
and $S_i=\{s_i\} \cup \{x \in V(G) : x \ \text{is a twin of} \
s_i\}$ for $i=1,2$. Then, $2\leq \mathit{wtn}(G) \leq |S_1 \cup
S_2|.$
\end{corollary}

\section{Conclusions}

In this work, we introduced a new graph convexity based on the
concept of weakly toll walks, which generalize induced paths, and
showed how such a convexity gives rise to a new structural
characterization of proper interval graphs. Also, we found bounds
for the weakly toll number and the weakly toll hull number of an
arbitrary interval graph.

We propose a further study of these two invariants in general
graphs, namely, determining the graphs $G$ for which
$\mathit{wtn}(G)=\mathit{wth}(G)=|\mathit{Ext}(G)|$. Other
research direction is the study of the Carath\'eodory number, the
Radon number, and the Helly number in the context of the weakly
toll convexity. Finally, characterizing weakly toll convex sets in
graph products is also an interesting open problem.

\if 10

\fi



\acknowledgements \label{sec:ack} M. C. Dourado is partially
supported by Conselho Nacional de Desenvolvimento Cient\'{\i}fico
e Tecnol\'{o}gico (CNPq), Brazil, Grant number 305404/2020-2 and
FAPERJ (211.753/2021). F. Protti is partially supported by CNPq
(304117/2019-6) and FAPERJ (201.083/2021).

\bibliographystyle{abbrvnat}
\bibliography{bib-dmtcs}

\end{document}